\documentclass[reqno]{mfatshort}
\usepackage{cite}

\newtheorem{theorem}{Theorem}[section]

\newtheorem{lemma}{Lemma}[section]

\theoremstyle{remark}
\newtheorem{remark}{Remark}[section]

\theoremstyle{definition}

\numberwithin{equation}{section}

\begin{document}

\title[Petrovskii parabolic systems in generalized Sobolev spaces]
{Some problems for Petrovskii parabolic systems\\ in generalized Sobolev spaces}

\author[Aleksandr Dyachenko, Valerii Los]{Aleksandr Dyachenko, Valerii Los}

\address{Aleksandr Dyachenko \newline National Technical University of Ukraine Igor Sikorsky Kyiv Polytechnic Institute,
Prospect Peremohy 37, 03056, Kyiv-56, Ukraine}
\email{ol\_v\_dyachenko@ukr.net}

\address{Valerii Los 
\newline National Technical University of Ukraine Igor Sikorsky Kyiv Polytechnic Institute,
Prospect Peremohy 37, 03056, Kyiv-56, Ukraine}
\email{v$\underline{\phantom{k}}$\,los@yahoo.com}


\date{29/03/2021}


\begin{abstract}
We consider an inhomogeneous initial-boundary value problem for a Petrovskii parabolic system of second order PDEs. We prove that this problem induces isomorphisms between appropriate anisotropic generalized Sobolev spaces. The regularity of these spaces are given by a pair of real numbers and by a function parameter. The latter allows us to characterize the regularity of solutions to the problem more finely as compared with anisotropic Sobolev spaces.
\end{abstract}

\maketitle

\section{Introduction}\label{sec1}
The central result of the theory of parabolic initial-boundary value problems states that these problems are well posed in the sense of
Hadamard on appropriate pairs of some normed function spaces.
This theory generally uses anisotropic versions of classical H\"older spaces or Sobolev spaces (see, e.g., \cite{AgranovichVishik64, Solonnikov65, LadyzhenskajaSolonnikovUraltzeva67, LionsMagenes72ii, Eidelman94,ZhitarashuEidelman98}). In recent years, other classes of function spaces are applied more and more actively to parabolic problems; namely, spaces with mixed norms, Triebel--Lizorkin spaces, weighted spaces, spaces of generalized smoothness (see, e.g., \cite{Weidemaier05, Hummel21, DenkHieberPruess07, Lindemulder20, DongKim15, LosMikhailetsMurach17CPAA} and reference therein). Our paper continues the series of works
\cite{LosMikhailetsMurach17CPAA, Los16UMJ6, LosMurach17OpenMath, Los16UMJ9, Los16UMJ11, LosMikhailetsMurach21arXiv} aimed to build a theory of parabolic initial-boundary value problems in function spaces of generalized anisotropic smoothness. This smoothness is given by a pair of real numbers and by a radial function which varies slowly at infinity and characterizes supplementary smoothness with respect to that given by the numbers. Such spaces give a broad generalization of anisotropic versions of inner product Sobolev spaces applied usually to parabolic equations. These works have elaborated a relevant theory of solvability and maximal regularity for scalar parabolic problems (involving one PDE). The present paper considers a general enough inhomogeneous initial-boundary volume problem for a Petrovskii parabolic system of second-order PDEs. We prove that this problem is well posed on appropriate pairs of generalized anisotropic Sobolev spaces just mentioned. This result is formulated in terms of an isomorphism theorem and is proved with the help of the quadratic interpolation (with function parameter) between some anisotropic Sobolev spaces. The special case of homogeneous Cauchy data was studied in \cite{Los17UMJ3, Los20MFAT2}. Note that various spaces of generalized smoothness proved to be useful in the theory of PDEs \cite{Hermander63, Hermander83, MikhailetsMurach14, NicolaRodino10, Paneah00} and the theory of stochastic processes \cite{Jacob010205}. Specifically, the monograph \cite{MikhailetsMurach14} presents a theory of elliptic boundary-value problems for isotropic analogs of the spaces used in the present paper (see also \cite{AnopDenkMurach21CPAA, AnopKasirenko16MFAT} for elliptic problems in broader classes of generalized Sobolev spaces).

\section{Statement of the problem}\label{sec2}
We arbitrarily choose an integer $n\geq2$ and a real number $\tau>0$. Let $G$ be a bounded domain in $\mathbb{R}^{n}$ with an infinitely smooth boundary $\Gamma:=\partial G$. We put
$\Omega:=G\times(0,\tau)$ and $S:=\Gamma\times(0,\tau)$; thus, $\Omega$ is an open cylinder in $\mathbb{R}^{n+1}$, and $S$ is its lateral boundary. Then $\overline{\Omega}:=\overline{G}\times[0,\tau]$ and
$\overline{S}:=\Gamma\times[0,\tau]$ are the closures of $\Omega$ and $S$ respectively.

Consider the following parabolic initial-boundary-value problem in $\Omega$:
\begin{equation}\label{25f1}
\begin{split}
\partial_t u_j(x,t)+
\sum_{k=1}^{N}\sum_{|\alpha|\leq2}&a^{\alpha}_{j,k}(x,t)\,D^\alpha_x u_{k}(x,t)=f_{j}(x,t)\\
&\mbox{for all}\quad (x,t)\in\Omega\quad\mbox{and}\quad
j\in\{1,\dots,N\};
\end{split}
\end{equation}
\begin{equation}\label{25f2}
\begin{split}
\sum_{k=1}^{N}\sum_{|\alpha|\leq l_j}&b^{\alpha}_{j,k}(x,t)\,D^\alpha_x u_{k}(x,t)\big|_{S}=g_{j}(x,t)\\
&\mbox{for all}\quad (x,t)\in S\quad\mbox{and}\quad
j\in\{1,\dots,N\};
\end{split}
\end{equation}
\begin{equation}\label{25f3}
u_{j}(x,t)\big|_{t=0}=h_{j}(x)\quad
\mbox{for all}\quad x\in G\quad\mbox{and}\quad j\in\{1,\ldots,N\}.
\end{equation}
Here, the integer $N\geq2$ is arbitrarily chosen,
and  $l_1,\ldots,l_N\in\{0,\,1\}$.
All coefficients of the partial differential expressions in formulas \eqref{25f1} and \eqref{25f2} are supposed to be infinitely smooth complex-valued functions given on $\overline{\Omega}$ and $\overline{S}$ respectively; i.e., each
\begin{equation*}
a_{j,k}^{\alpha,\beta}\in C^{\infty}(\overline{\Omega}):=
\bigl\{w\!\upharpoonright\overline{\Omega}\!:\,w\in C^{\infty}(\mathbb{R}^{n+1})\bigr\}
\end{equation*}
and each
\begin{equation*}
b_{j,k}^{\alpha,\beta}\in C^{\infty}(\overline{S}):=
\bigl\{v\!\upharpoonright\overline{S}\!:\,v\in C^{\infty}(\Gamma\times\mathbb{R})\bigr\}.
\end{equation*}
We use the notation
$D^\alpha_x:=D^{\alpha_1}_{1}\dots D^{\alpha_n}_{n}$, with $D_{k}:=i\,\partial/\partial{x_k}$, and $\partial_t:=\partial/\partial t$
for the partial derivatives of functions depending on $x=(x_1,\ldots,x_n)\in\mathbb{R}^{n}$ and $t\in\mathbb{R}$. Here, $i$ is imaginary unit, and $\alpha=(\alpha_1,...,\alpha_n)$ is a multi-index, with $|\alpha|:=\alpha_1+\cdots+\alpha_n$. In formulas \eqref{25f1} and \eqref{25f2} and their analogs, we take summation over the integer-valued nonnegative indices $\alpha_1,...,\alpha_n$ that satisfy the condition written under the integral sign.

We assume that the initial-boundary value problem \eqref{25f1}--\eqref{25f3} is Petrovskii parabolic in the cylinder $\Omega$. Let us recall the corresponding definition \cite[Section~1, \S~1]{Solonnikov65}.
Let
\begin{equation}\label{25f4}
A_{j,k}(x,t,D_x,\partial_t):=\delta_{j,k}\partial_t+
\sum_{|\alpha|\leq 2}a^{\alpha}_{j,k}(x,t)\,D^\alpha_x
\end{equation}
and
\begin{equation}\label{25f5}
B_{j,k}(x,t,D_x):=\sum_{|\alpha|\leq l_{j}}b^{\alpha}_{j,k}(x,t)\,D^\alpha_x
\end{equation}
for all $j,k\in\{1,\ldots,N\}$. Here, $\delta_{j,k}$ is the Kronecker delta.
Using notation \eqref{25f4} and \eqref{25f5} we may rewrite
all equalities in \eqref{25f1} and \eqref{25f2} as follows:
\begin{equation*}
\sum_{k=1}^{N}A_{j,k}(x,t,D_x,\partial_t)u_{k}(x,t)=f_{j}(x,t)
\end{equation*}
and
\begin{equation*}
\sum_{k=1}^{N}B_{j,k}(x,t,D_x)u_{k}(x,t)\big|_{S}=
g_{j}(x,t).
\end{equation*}
Define the principal symbols of the linear partial differential operators \eqref{25f4} and \eqref{25f5} by the formulas
\begin{equation*}
A^{(0)}_{j,k}(x,t,\xi,p):=\delta_{j,k}p+
\sum_{|\alpha|=2}a^{\alpha}_{j,k}(x,t)\,\xi^\alpha
\end{equation*}
and
\begin{equation*}
B^{(0)}_{j,k}(x,t,\xi):=\sum_{|\alpha|=
l_{j}}b^{\alpha}_{j,k}(x,t)\,\xi^\alpha.
\end{equation*}
These symbols are homogeneous polynomials in
$\xi:=(\xi_{1},\ldots,\xi_{n})\in\mathbb{C}^{n}$ and $p\in\mathbb{C}$ jointly (as usual, $\xi^\alpha:=\xi_{1}^{\alpha_{1}}\ldots\xi_{n}^{\alpha_{n}}$).
Consider the matrices
\begin{equation*}
A^{(0)}(x,t,\xi,p):=\bigl(A^{(0)}_{j,k}(x,t,\xi,p)\bigr)_{j,k=1}^{N}\\
\end{equation*}
and
\begin{equation*}
B^{(0)}(x,t,\xi):=\bigl(B^{(0)}_{j,k}(x,t,\xi)\bigr)_{j,k=1}^{N}.
\end{equation*}
The problem \eqref{25f1}--\eqref{25f3} is said to be Petrovskii parabolic in $\Omega$ if it satisfies the following two conditions (i) and (ii):
\begin{itemize}
\item [(i)] For arbitrary points $x\in\overline{G}$ and $t\in[0,\tau]$ and every vector $\xi\in\mathbb{R}^{n}$, all the roots $p(x,t,\xi)$ of the polynomial $\det A^{(0)}(x,t,\xi,p)$ in $p\in\mathbb{C}$ satisfy the inequality $\mathrm{Re}\,p(x,t,\xi)\leq -\delta\,|\xi|^{2b}$ for some number $\delta>0$ that does not depend on $x$, $t$, and~$\xi$.
\end{itemize}

To formulate Condition (ii), we fix a number $\delta_1\in(0,\delta)$, where $\delta$ has appeared in Condition~(i), and then arbitrarily choose a point $x\in\Gamma$, real number $t\in[0,\tau]$, vector $\xi\in\mathbb{R}^{n}$ tangent to the boundary $\Gamma$ at $x$, and number $p\in\mathbb{C}$ such that $\mathrm{Re}\,p\geq -\delta_1|\xi|^{2b}$ and $|\xi|+|p|\neq0$. Let $\nu(x)$ denote the unit vector of the inward normal to $\Gamma$ at $x$. It follows from Condition~(i) and the inequality $n\geq2$ that the polynomial
$\det A^{(0)}(x,t,\xi+\zeta\nu(x),p)$ in $\zeta\in\mathbb{C}$ has $m$ roots $\zeta^{+}_{j}(x,t,\xi,p)$, $j=\nobreak1,\ldots,m$, with positive imaginary part and $m$ roots with negative imaginary part, taking into account multiplicity of roots.

The second condition is formulated as follows:
\begin{itemize}
\item [(ii)] For some positive number $\delta_1<\delta$ and for every choice of the parameters $x$, $t$, $\xi$ and $p$ indicated above, the rows of the matrix
\begin{equation*}
B^{(0)}(x,t,\xi+\zeta\nu(x),p)\cdot\widetilde{A}^{(0)}
(x,t,\xi+\zeta\nu(x),p)
\end{equation*}
are linearly independent modulo the polynomial
$\prod_{j=1}^{m}(\zeta-\zeta^{+}_{j}(x,t,\xi,p))$. Here,  $\widetilde{A}^{(0)}$ is the transposed matrix of the cofactors of entries of $A^{(0)}$.
\end{itemize}

We investigate parabolic problem \eqref{25f1}--\eqref{25f3}
in appropriate generalized Sobolev inner product spaces considered in the next section. All functions (and distributions) are supposed to be complex-valued.

\section{Generalized Sobolev spaces related to the problem}\label{sec3}
Following \cite[Section~3]{LosMurach17OpenMath} (see also \cite[Section~3]{LosMikhailetsMurach17CPAA}), we consider generalized Sobolev spaces used for the investigation of the parabolic problem \eqref{25f1}--\eqref{25f3}. They are parameterized with two numbers $s$ and $s/2$, where $s\in\mathbb{R}$, and with a function $\varphi\in\mathcal{M}$. The class $\mathcal{M}$ is defined to consist of all Borel measurable functions
$\varphi:[1,\infty)\rightarrow(0,\infty)$ such that
\begin{itemize}
\item [($\ast$)] both the functions $\varphi$ and $1/\varphi$ are bounded on each
compact interval $[1,d]$, with $1<d<\infty$;
\item [($\ast\ast$)] $\varphi$ is a slowly varying function at infinity in the sense of J.~Karamata \cite{Karamata30a}; i.e.,
\begin{equation}\label{25f3.1}
\lim_{r\rightarrow\infty}\frac{\varphi(\lambda r)}{\varphi(r)}=1\quad\mbox{for
every}\quad \lambda>0.
\end{equation}
\end{itemize}
The theory of slowly varying functions is set forth in \cite{BinghamGoldieTeugels89}. We give an important and standard example of functions satisfying \eqref{25f3.1} by putting
\begin{equation}\label{25f3.2}
\varphi(r):=(\log r)^{\theta_{1}}\,(\log\log r)^{\theta_{2}} \ldots
(\,\underbrace{\log\ldots\log}_{k\;\mbox{\small{times}}}r\,)^{\theta_{k}}
\quad\mbox{for}\quad r\gg1,
\end{equation}
where the parameters $k\in\mathbb{N}$ and $\theta_{1},
\theta_{2},\ldots,\theta_{k}\in\mathbb{R}$ are chosen arbitrarily. The functions
\eqref{25f3.2} form the logarithmic multiscale, which has a number of applications in
the theory of function spaces. Some other examples of slowly varying functions can
be found in \cite[Sec. 1.3.3]{BinghamGoldieTeugels89} and \cite[Sec.
1.2.1]{MikhailetsMurach14}.

By definition, the (complex) linear space $H^{s,s/2;\varphi}(\mathbb{R}^{k})$, where $2\leq k\in\mathbb{Z}$, consists of all tempered distributions $w$ on $\mathbb{R}^{k}$ whose (complete) Fourier transform $\widetilde{w}$ is locally Lebesgue integrable over $\mathbb{R}^{k}$ and satisfies the condition
\begin{equation}\label{norm}
\|w\|_{H^{s,s/2;\varphi}(\mathbb{R}^{k})}:=
\biggl(\;\int\limits_{\mathbb{R}^{k-1}}\int\limits_{\mathbb{R}}
r^{2s}(\xi,\eta)\,\varphi^{2}(r(\xi,\eta))\,
|\widetilde{w}(\xi,\eta)|^{2}\,d\xi\,d\eta\biggr)^{1/2}<\infty,
\end{equation}
where
$$
r(\xi,\eta):=\bigl(1+|\xi|^2+|\eta|\bigr)^{1/2}
\quad\mbox{for all}\;\;\xi\in\mathbb{R}^{k-1}\;\;\mbox{and}
\;\;\eta\in\mathbb{R}.
$$
This space is Hilbert and separable with respect to the norm \eqref{norm}.

It is a special case of the spaces $\mathcal{B}_{p,\mu}$ introduced by H\"ormander \cite[Section~2.2]{Hermander63}; namely, $H^{s,s/2;\varphi}(\mathbb{R}^{k})=
\mathcal{B}_{p,\mu}$ provided that $p=2$ and $\mu(\xi,\eta)\equiv r^{s}(\xi,\eta)\varphi(r(\xi,\eta))$.
The space $H^{s,s/2;\varphi}(\mathbb{R}^{k})$ give a broad generalization of the concept of Sobolev spaces (in the framework of Hilbert spaces).
If $\varphi(\cdot)\equiv1$, the space $H^{s,s/2;\varphi}(\mathbb{R}^{k})$ becomes the anisotropic Sobolev space $H^{s,s/2}(\mathbb{R}^{k})$.
Generally, we have the dense continuous embeddings
\begin{equation}\label{embeddings}
H^{s_{1},s_{1}/2}(\mathbb{R}^{k})\hookrightarrow
H^{s,s/2;\varphi}(\mathbb{R}^{k})\hookrightarrow
H^{s_{0},s_{0}/2}(\mathbb{R}^{k})\quad
\mbox{whenever}\quad s_{0}<s<s_{1}.
\end{equation}

Basing on $H^{s,s/2;\varphi}(\mathbb{R}^{k})$, consider some Hilbert function spaces  relating to the problem \eqref{25f1}--\eqref{25f3}. Let $V$ be an open nonempty set in $\mathbb{R}^{k}$. (Specifically, we need the case where
$V=\Omega$, with $k=n+1$.) Put
\begin{equation}\label{25f7}
H^{s,s/2;\varphi}(V):=\bigl\{w\!\upharpoonright\!V:\,
w\in H^{s,s/2;\varphi}(\mathbb{R}^{k})\bigl\}.
\end{equation}
The linear space \eqref{25f7} is endowed with the norm
\begin{equation}\label{25f8}
\|u\|_{H^{s,s/2;\varphi}(V)}:=
\inf\bigl\{\,\|w\|_{H^{s,s/2;\varphi}(\mathbb{R}^{k})}:
w\in H^{s,s/2;\varphi}(\mathbb{R}^{k}),\;\,
u=w\!\upharpoonright\!V\bigl\},
\end{equation}
where $u\in H^{s,s/2;\varphi}(V)$. This space is Hilbert and separable with respect to this norm. The set
$C^{\infty}(\overline{\Omega})$ is dense in $H^{s,s/2;\varphi}(\Omega)$.

We also need the space $H^{s,s/2;\varphi}(S)$ on the lateral boundary $S$ of the cylinder $\Omega$ (see \cite[Section~1]{Los16JMathSci}). We restricting ourselves to the $s>0$ case. Briefly saying, this space consists of all functions $v\in L_2(S)$ that yield functions from the space $H^{s,s/2;\varphi}(\Pi)$ on $\Pi:=\mathbb{R}^{n-1}\times(0,\tau)$ with the help of some local coordinates on $\overline{S}$. Let us recall the detailed definition.

We arbitrarily choose a finite atlas on $\Gamma$ of class $C^{\infty}$.
Let this atlas be formed by some local charts $\theta_{j}:\mathbb{R}^{n-1}\leftrightarrow\Gamma_{j}$, with
$j=1,\ldots,\lambda$. Here, $\Gamma_{1},\ldots,\Gamma_{\lambda}$ are open nonempty subsets of $\Gamma$ such that $\Gamma:=\Gamma_{1}\cup\cdots\cup\Gamma_{\lambda}$. We also arbitrarily choose functions
$\chi_{j}\in C^{\infty}(\Gamma)$, with $j=1,\ldots,\lambda$, such that
$\mathrm{supp}\,\chi_{j}\subset\Gamma_{j}$ and $\chi_{1}+\cdots+\chi_{\lambda}=1$ on $\Gamma$. Thus, these functions form a partition of unity on $\Gamma$.

By definition, the complex linear space $H^{s,s/2;\varphi}(S)$ consists of all functions $v\in L_2(S)$ such that the function $v_{j}(y,t):=\chi_{j}(\theta_{j}(y))v(\theta_{j}(y),t)$ of $y\in\mathbb{R}^{n-1}$ and $t>0$ belongs to $H^{s,s/2;\varphi}(\Pi)$ for each $j\in\{1,\ldots,\lambda\}$. (As usual, $L_2(S)$ denotes the space of all square integrable functions on the surface $S$.) The space $H^{s,s/2;\varphi}(S)$ is endowed with the norm
$$
\|v\|_{H^{s,s/2;\varphi}(S)}:=
\bigl(\|v_{1}\|_{H^{s,s/2;\varphi}(\Pi)}^{2}+\cdots+
\|v_{\lambda}\|_{H^{s,s/2;\varphi}(\Pi)}^{2}\bigr)^{1/2}.
$$
This space is Hilbert and separable with respect to this norm and does not depend up to equivalence of norms on the indicated choice of an atlas and partition of unity on~$\Gamma$ \cite[Theorem~1]{Los16JMathSci}.

We also need the isotropic generalized Sobolev spaces $H^{s;\varphi}(\mathbb{R}^{k})$, $H^{s;\varphi}(G)$, and $H^{s;\varphi}(\Gamma)$ investigated by Mikhailets and Murach \cite{MikhailetsMurach14, MikhailetsMurach12BJMA2}. Let $s\in\mathbb{R}$ and $\varphi\in\mathcal{M}$.
By definition, the complex linear space $H^{s;\varphi}(\mathbb{R}^{k})$, where $1\leq k\in\mathbb{Z}$, consists of all tempered distributions $w$ on $\mathbb{R}^{k}$ whose (complete) Fourier transform $\widetilde{w}$ is locally Lebesgue integrable over $\mathbb{R}^{k}$ and satisfies the condition
\begin{equation}\label{norm-isotrop}
\|w\|_{H^{s;\varphi}(\mathbb{R}^{k})}:=
\biggl(\;\int\limits_{\mathbb{R}^{k}}
\langle\xi\rangle^{2s}\,\varphi^{2}(\langle\xi\rangle)\,
|\widetilde{w}(\xi)|^{2}\,d\xi\biggr)^{1/2}<\infty.
\end{equation}
Here, as usual, $\langle\xi\rangle:=(1+|\xi|^{2})^{1/2}$ is the smooth modulus of
$\xi\in\mathbb{R}^{k}$.
This space is Hilbert and separable with respect to the norm \eqref{norm-isotrop}.
Notice that $H^{s;\varphi}(\mathbb{R}^{k})$ is the inner product space
$\mathcal{B}_{2,\mu}(\mathbb{R}^{k})$ corresponding to the function parameter
$\mu(\xi):=\langle\xi\rangle^{s}\varphi(\langle\xi\rangle)$ of $\xi\in\mathbb{R}^{k}$.
The space $H^{s;\varphi}(\mathbb{R}^{k})$ is isotropic because the function $\mu(\xi)$ depends only on $|\xi|$.
If $\varphi(\cdot)\equiv1$, the space $H^{s;\varphi}(\mathbb{R}^{k})$ becomes the isotropic Sobolev space $H^{s}(\mathbb{R}^{k})$.

Basing on $H^{s;\varphi}(\mathbb{R}^{k})$, consider some Hilbert function spaces on $G$ and $\Gamma$.
Let $V$ be an open nonempty set in $\mathbb{R}^{k}$. (Specifically, we need the case where
$V=G$, with $k=n$.) Put
\begin{equation}\label{25f7A}
H^{s;\varphi}(V):=\bigl\{w\!\upharpoonright\!V:\,
w\in H^{s;\varphi}(\mathbb{R}^{k})\bigl\}.
\end{equation}
The norm in the linear space \eqref{25f7A} is defined by the formula
\begin{equation}\label{25f8A}
\|h\|_{H^{s;\varphi}(V)}:=\,
\inf\bigl\{\,\|w\|_{H^{s;\varphi}(\mathbb{R}^{k})}:
w\in H^{s;\varphi}(\mathbb{R}^{k}),\;\,
h=w\!\upharpoonright\!V\bigl\},
\end{equation}
with $h\in H^{s;\varphi}(V)$. This space is Hilbert and separable with respect to this norm.

The linear space $H^{s;\varphi}(\Gamma)$ is defined to consist of all distributions $\omega$ on $\Gamma$ such that every distribution
$\omega_{j}(y):=\chi_{j}(\theta_{j}(y))\,\omega(\theta_{j}(y))$ of
$y\in\mathbb{R}^{n-1}$ belongs to $H^{s;\varphi}(\mathbb{R}^{n-1})$, where $j\in\{1,\ldots,\lambda\}$. The space $H^{s;\varphi}(\Gamma)$ is equipped with the norm
\begin{equation*}
\|\omega\|_{H^{s;\varphi}(\Gamma)}:=\bigl(
\|\omega_{1}\|_{H^{s;\varphi}(\mathbb{R}^{n-1})}^{2}+\cdots+
\|\omega_{\lambda}\|_{H^{s;\varphi}(\mathbb{R}^{n-1})}^{2}\bigr)^{1/2}.
\end{equation*}
This space is Hilbert and separable and is independent (up to equivalence of norms) of our choice of local charts and partition  of the unit on $\Gamma$ (see \cite[Theorem 2.1]{MikhailetsMurach14}).

If $\varphi\equiv1$, the spaces $H^{s,s/2;\varphi}(\cdot)$ and $H^{s;\varphi}(\cdot)$ become the inner product Sobolev spaces $H^{s,s/2}(\cdot)$ and $H^{s}(\cdot)$, respectively. Owing to \eqref{embeddings}, we have the embeddings
\begin{equation}\label{25f5a}
H^{s_{1},s_{1}/2}(\cdot)\hookrightarrow
H^{s,s/2;\varphi}(\cdot)\hookrightarrow
H^{s_{0},s_{0}/2}(\cdot)\quad\mbox{whenever}\quad s_{0}<s<s_{1}.
\end{equation}
Besides,
\begin{equation}\label{25f5b}
H^{s_{1}}(\cdot)\hookrightarrow
H^{s;\varphi}(\cdot)\hookrightarrow
H^{s_{0}}(\cdot)\quad\mbox{whenever}\quad s_{0}<s<s_{1},
\end{equation}
as is noted in \cite[Theorems 2.3(iii) and 3.3(iii)]{MikhailetsMurach14}. All these embeddings are continuous and dense.

Generally, if $\varphi(\cdot)\equiv1$, we will suppress the index $\varphi$ in the designations of relevant spaces.

\section{Main result}\label{sec4}

We  will formulate an isomorphism theorem for the parabolic problem \eqref{25f1}--\eqref{25f3} in the generalized Sobolev spaces introduced in pervious section. Considering this problem, we put $u:=(u_1,\ldots,u_N)$, $f:=(f_1,\ldots,f_N)$, $g:=(g_1,\ldots,g_N)$, and $h:=(h_1,\ldots,h_N)$.
We write the system \eqref{25f1} and the boundary conditions \eqref{25f2} in the matrix form
$Au=f$ and $Bu|_S=g$; here
$$
A:=(A_{j,k}(x,t,D_x,\partial_t))_{j,k=1}^N\quad\mbox{and}\quad
B:=\bigl(B_{j,k}(x,t,D_x)\bigr)_{j,k=1}^N
$$
are matrix differential operators. We associate the linear mapping
\begin{equation}\label{25f4.1}
u\mapsto\Lambda
u:=\bigl(Au,Bu,u\!\upharpoonright\!G\bigr),\quad\mbox{where}\;\,u\in
\bigl(C^{\infty}(\overline{\Omega})\bigr)^{N},
\end{equation}
with the problem \eqref{25f1}--\eqref{25f3}.

In order that a regular enough solution $u$ to the problem \eqref{25f1}--\eqref{25f3} exist, the right-hand sides to the problem should satisfy certain compatibility conditions (see, e.g., \cite[Section~14]{Solonnikov65}).
These conditions consist in that the partial derivatives $\partial^r_t u_j(x,t)\big|_{t=0}$, which can be found from the parabolic system \eqref{25f1} and initial conditions \eqref{25f3}, should satisfy the boundary conditions \eqref{25f2} and some relations obtained by the differentiation of the boundary conditions with respect to~$t$. To write these compatibility conditions, we consider problem \eqref{25f1}--\eqref{25f3} in the corresponding pairs of Sobolev spaces.

Let real $s\geq2$. The mapping \eqref{25f4.1} extends uniquely (by continuity)
to a bounded linear operator
\begin{equation}\label{25f4a}
\begin{split}
\Lambda:\:& \bigl(H^{s,s/2}(\Omega)\bigr)^N\\
&\leftrightarrow
\bigl(H^{s-2,s/2-1}(\Omega)\bigr)^N\oplus
\bigoplus_{j=1}^{N}H^{s-l_j-1/2,(s-l_j-1/2)/2}(S)\oplus\bigl(H^{s-1}(G)\bigr)^N.
\end{split}
\end{equation}
This follows directly from \cite[Chapter~I, Lemma~4, and Chapter~II, Theorems 3 and 7]{Slobodetskii58}.

Choosing any vector-function $u(x,t)$ from the space $(H^{s,s/2}(\Omega))^N$, we define the right-hand sides
\begin{equation}\label{25f4b}
\begin{gathered}
f\in\bigl(H^{s-2,s/2-1}(\Omega)\bigr)^N,\quad \\
g\in \bigoplus_{j=1}^{N}H^{s-l_j-1/2,(s-l_j-1/2)/2}(S),
\quad\mbox{and}\quad h\in\bigl(H^{s-1}(G)\bigr)^N
\end{gathered}
\end{equation}
of the problem by the formula
$(f,g,h):=\Lambda u$
with the help of the bounded operator \eqref{25f4a}.

Compatibility conditions for functions $f_j$, $g_j$, and $h_j$ naturally arise in such a way.
According to \cite[Chapter~II, Theorem 7]{Slobodetskii58}, the traces
$\partial^{\,r}_t u_j(\cdot,0)\in H^{s-2r-1}(G)$ (for all $j\in\{1,\dots,N\}$) are well defined by closure for all $r\in\mathbb{Z}$ such that $0\leq r<s/2-1/2$ (and only for these $r$).
These traces are expressed
from the equations \eqref{25f1} and initial conditions \eqref{25f3}
in terms of the functions $f_j$ and $h_j$ by recurrent formula:
\begin{equation}\label{25f6}
\begin{split}
u_j(x,0)&=h_j(x)\quad\mbox{for all}
\quad j\in\{1,\dots,N\},\\
(\partial^{\,r}_t u_j)(x,0)&=-\sum_{k=1}^{N}
\sum_{|\alpha|\leq 2}
\sum\limits_{q=0}^{r-1}
\binom{r-1}{q}(\partial^{\,r-1-q}_t
a_{j,k}^{\alpha})(x,0)\,D^\alpha_x(\partial^{q}_t
u_k)(x,0)+\\
&+\partial^{\,r-1}_t f_j(x,0)\\
&\qquad\mbox{for all}\quad j\in\{1,\dots,N\}\quad\mbox{and}\\
&\qquad\mbox{for each}\quad r\in\mathbb{Z}\quad\mbox{such that}\quad 1\leq r<s/2-1/2.
\end{split}
\end{equation}
These equalities hold true for almost all $x\in G$.

Besides, according to \cite[Chapter~II, Theorem 7]{Slobodetskii58}, for each $j\in\{1,\dots,N\}$
the traces $\partial^{\,r}_t g_j(\cdot,0)\in H^{s-l_j-3/2-2r}(\Gamma)$
are well defined by closure for all $r\in\mathbb{Z}$
such that $0\leq r<(s-l_j-3/2)/2$ (and only for these $r$).
We express these traces in terms of the functions $u_j(x,t)$ and its time derivatives by the formula
\begin{equation}\label{25f4bb}
\begin{aligned}
(\partial^{r}_t g_j)(x,0)&=\partial^{r}_t\biggl(\sum_{k=1}^{N}\sum_{|\alpha|\leq l_j}b^{\alpha}_{j,k}(x,t)\,D^\alpha_x u_{k}(x,t)\biggr)|_{t=0}\\
&=\sum_{k=1}^{N}\sum_{|\alpha|\leq l_j}\,
\sum_{q=0}^{r}\binom{r}{q}
(\partial^{r-q}_t b^{\alpha}_{j,k})(x,0)\,
D^\alpha_x(\partial^{q}_t u_k)(x,0)
\end{aligned}
\end{equation}
for almost all $x\in\Gamma$. Here, all the functions
$(\partial^{\,q}_{t}u_k)(x,0)$
of $x\in G$ are expressed in terms of the functions $f_j(x,t)$ and $h_{j}(x)$ by the recurrent formula \eqref{25f6}.

Substituting \eqref{25f6} in the right-hand side of formula \eqref{25f4bb}, we obtain the compatibility conditions
\begin{equation}\label{25f8}
\begin{gathered}
\partial^{r}_t g_j\!\upharpoonright\!\Gamma=
\mathcal{B}_{j,r}(v_{1,0},\dots,v_{N,0},\dots,v_{1,r},\dots,v_{N,r})\!\upharpoonright\!\Gamma \\
\mbox{for each}\;\;j\in\{1,\dots,N\}\;\;\mbox{and}\;\;r\in\mathbb{Z}\\
\mbox{such that}\;\;0\leq r<(s-l_j-3/2)/2.
\end{gathered}
\end{equation}
Here,
the functions $v_{1,0},v_{2,0},\ldots$ are defined almost everywhere on $G$  by the formulas
\begin{equation}\label{25f9}
\begin{aligned}
v_{j,0}(x)&=h_j(x),\\
v_{j,r}(x)&=-\sum_{k=1}^{N}
\sum_{|\alpha|\leq 2}
\sum\limits_{q=0}^{r-1}
\binom{r-1}{q}(\partial^{\,r-1-q}_t
a_{j,k}^{\alpha})(x,0)\,D^\alpha_x v_{k,q}(x,0)+\\
&\qquad\quad+\partial^{\,r-1}_t f_j(x,0)
\quad\mbox{if}\quad r\geq1,
\end{aligned}
\end{equation}
and
\begin{equation}\label{25f9B}
\begin{aligned}
&\mathcal{B}_{j,r}(v_{1,0},\dots,v_{N,0},\dots,v_{1,r},\dots,v_{N,r})(x)\\
&=\sum_{k=1}^{N}\sum_{|\alpha|\leq l_j}\,
\sum_{q=0}^{r}\binom{r}{q}
(\partial^{r-q}_t b^{\alpha}_{j,k})(x,0)\,
D^\alpha_x v_{k,q}(x,0)
\end{aligned}
\end{equation}
for almost all $x\in G$. Note, that
\begin{equation*}
v_{j,r}\in H^{s-2r-1}(G)\quad\mbox{for each}\quad
r\in\mathbb{Z}\cap[0,s/2-1/2)
\end{equation*}
due to \eqref{25f4b}.
The right-hand side of the equality in \eqref{25f8} is well defined because the function
$\mathcal{B}_{j,r}(v_{1,0},\dots,v_{N,0},\dots,v_{1,r},\dots,v_{N,r})$ belongs to $H^{s-l_j-2r-1}(G)$ and the trace
\begin{equation}\label{25f69aa}
\mathcal{B}_{j,r}(v_{1,0},\dots,v_{N,0},\dots,v_{1,r},\dots,v_{N,r})\!\upharpoonright\!\Gamma\in H^{s-l_j-2r-3/2}(\Gamma)
\end{equation}
is therefore defined by closure whenever $s-l_j-2r-3/2>0$.

The number of the compatibility conditions \eqref{25f8} is a function of
$s\geq2$. This function is discontinuous at $s$ if and only if
$(s-l_j-3/2)/2\in\mathbb{Z}$. Thus, the set of all its discontinuities
coincides with
\begin{equation}\label{setE}
E:=\{2l+l_j+3/2:j,l\in\mathbb{Z},\;1\leq j\leq N,\;l\geq0\}
\cap(2,\infty).
\end{equation}
Note, if $s\leq5/2$ and $l_j=1$ for some $j$, there are no compatibility conditions involving~$g_{j}$.

Our result on the parabolic problem \eqref{25f1}--\eqref{25f3} consists in that the linear mapping \eqref{25f4.1} extends uniquely  to an isomorphism between appropriate pairs of generalized Sobolev spaces introduced in the previous section. Let us indicate these pairs.
We arbitrarily choose a real number $s>2$ and function parameter $\varphi\in\mathcal{M}$.
We take $(H^{s,s/2;\varphi}(\Omega))^N$ as the domain of this isomorphism. Its range is imbedded in the Hilbert space
\begin{align*}
\mathcal{H}^{s-2,s/2-1;\varphi}:=
&\bigl(H^{s-2,s/2-1;\varphi}(\Omega)\bigr)^N\\
&\oplus\bigoplus_{j=1}^{N}H^{s-l_j-1/2,(s-l_j-1/2)/2;\varphi}(S)
\oplus\bigl(H^{s-1;\varphi}(G)\bigr)^N
\end{align*}
and is denoted by  $\mathcal{Q}^{s-2,s/2-1;\varphi}$.
[If $\varphi\equiv1$, then $\mathcal{H}^{s-2,s/2-1;\varphi}$ is the target space of \eqref{25f4a}.]
We separately define $\mathcal{Q}^{s-2,s/2-1;\varphi}$
in the cases where $s\notin E$ and where $s\in E$.

Suppose first that $s\notin E$. By definition, the linear space $\mathcal{Q}^{s-2,s/2-1;\varphi}$ consists of all vectors
$$
F:=\bigl(f_1,\dots,f_N,g_1,\dots,g_N,h_1,\dots,h_N\bigr)\in
\mathcal{H}^{s-2,s/2-1;\varphi}
$$
that satisfy the compatibility conditions \eqref{25f8}.
These conditions are well defined for every indicated $F$ because
they are well defined whenever
$F\in\mathcal{H}^{s-\varepsilon-2,s/2-\varepsilon/2-1}$ and
$0<\varepsilon\ll 1$ and because
\begin{equation}\label{25f69a}
\mathcal{H}^{s-2,s/2-1;\varphi}\hookrightarrow
\mathcal{H}^{s-\varepsilon-2,s/2-\varepsilon/2-1}.
\end{equation}
This continuous embedding  follows directly from \eqref{25f5a} and \eqref{25f5b}.
We endow the linear space $\mathcal{Q}^{s-2,s/2-1;\varphi}$ with the norm in the Hilbert space
$\mathcal{H}^{s-2,s/2-1;\varphi}$. The space $\mathcal{Q}^{s-2,s/2-1;\varphi}$
is complete, i.e. Hilbert. Indeed,
$$
\mathcal{Q}^{s-2,s/2-1;\varphi}=
\mathcal{H}^{s-2,s/2-1;\varphi}\cap
\mathcal{Q}^{s-\varepsilon-2,s/2-\varepsilon/2-1}
$$
whenever $0<\varepsilon\ll 1$.
Here, the space $\mathcal{Q}^{s-\varepsilon-2,s/2-\varepsilon/2-1}$ is complete because the differential operators and trace operators used in the compatibility conditions are bounded on the corresponding pairs of Sobolev spaces. Therefore, the space
$$
\mathcal{H}^{s-2,s/2-1;\varphi}\cap
\mathcal{Q}^{s-\varepsilon-2,s/2-\varepsilon/2-1}
$$
is complete with respect to the sum of the norms in the components of the intersection, this sum being equivalent to the norm in $\mathcal{H}^{s-2,s/2-1;\varphi}$ due to \eqref{25f69a}. Thus, the space $\mathcal{Q}^{s-2,s/2-1;\varphi}$ is complete (with respect to the latter norm).

If $s\in E$, then we define the Hilbert space $\mathcal{Q}^{s-2,s/2-1;\varphi}$
by means of the quadratic interpolation between its analogs just introduced. Namely, we put
\begin{equation}\label{25f10}
\mathcal{Q}^{s-2,s/2-1;\varphi}:=
\bigl[\mathcal{Q}^{s-\varepsilon-2,s/2-\varepsilon/2-1;\varphi},
\mathcal{Q}^{s+\varepsilon-2,s/2+\varepsilon/2-1;\varphi}\bigr]_{1/2}.
\end{equation}
Here, the number $\varepsilon\in(0,1/2)$ is arbitrarily chosen, and the right-hand side of the equality is the result of the quadratic interpolation with the parameter~$1/2$ of the written pair of Hilbert spaces.
The Hilbert space $\mathcal{Q}^{s-2,s/2-1;\varphi}$ defined by formula \eqref{25f10} does not depend on our choice of $\varepsilon$ up to equivalence of norms and is continuously embedded in $\mathcal{H}^{s-2,s/2-1;\varphi}$.
This will be shown in Remark~\ref{25rem8.1} at the end of Section~\ref{sec6}.

Now, we may formulate the main result of the paper.

\begin{theorem}\label{25th4.1}
For arbitrary $s>2$ and $\varphi\in\nobreak\mathcal{M}$ the mapping \eqref{25f4.1} extends
uniquely (by continuity) to an isomorphism
\begin{equation}\label{25f11}
\Lambda :\,\bigl(H^{s,s/2;\varphi}(\Omega)\bigr)^N\leftrightarrow
\mathcal{Q}^{s-2,s/2-1;\varphi}.
\end{equation}
\end{theorem}

This Theorem is known in the Sobolev case where $\varphi\equiv1$, which is proved in this case by Solonnikov
\cite[Theorem~5.4]{Solonnikov65} for general parabolic systems under the restriction  $s,s/2\in\mathbb{Z}$. This restriction can be removed, as is shown in Eidel'man and Zhitarashu's monograph \cite[Theorem~5.7]{ZhitarashuEidelman98}; their result includes the limiting case of $s=2$. Theorem~\ref{25th4.1} on isomorphisms has many applications.

Using this theorem, we can investigate global and local regularity of generalized solutions to the problem under consideration.
This theorem also allows us to obtain new sufficient conditions for the continuity of generalized solutions, specifically to find conditions under which the solutions are classical (c.f. \cite{LosMikhailetsMurach17CPAA, LosMikhailetsMurach21arXiv, Los16UMJ9, Los16UMJ11, Los17UMJ3, Los20MFAT2}). These applications will be given in another paper.

Note that the necessity to define the target space $\mathcal{Q}^{s-2,s/2-1;\varphi}$ separately in the  $s\in E$ case is caused by the following: if we defined this space for $s\in E$ in the way used in the $s\notin E$ case, then the isomorphism \eqref{25f11} would not be valid at least for $\varphi\equiv1$.
This follows from a result by Solonnikov \cite[Section~6]{Solonnikov64} (see also \cite[Remark~6.4]{LionsMagenes72ii}).

\section{Proof of the main result}\label{sec6}
We deduce Theorem \ref{25th4.1} from its known counterpart in the Sobolev case \cite[Theorem~5.7]{ZhitarashuEidelman98} with the help of the quadratic interpolation with a suitable function parameter. As to this interpolation, we mainly use definitions, notations, and properties given in \cite[Section~1.1]{MikhailetsMurach14}.
Let $X:=[X_{0},X_{1}]$ be an ordered pair of separable complex Hilbert spaces such that $X_1$ is a dense linear manifold in $X_0$ and that the
embedding $X_{1}\subseteq X_{0}$ is continuous. This pair is called regular. Let $[X_{0},X_{1}]_{\psi}$ denotes the Hilbert space, obtained by the quadratic interpolation with the function parameter $\psi$ of the pair $X=\nobreak[X_{0},X_{1}]$ or (in other words) between $X_{0}$ and $X_{1}$. We define the interpolation parameter $\psi$ by the formula
\begin{equation}\label{8f16}
\psi(r):=
\begin{cases}
\;r^{(s-s_{0})/(s_{1}-s_{0})}\,\varphi(r^{1/(s_{1}-s_{0})})&\text{if}
\quad r\geq1,\\
\;\varphi(1) & \text{if}\quad0<r<1,
\end{cases}
\end{equation}
where $s_{0},s,s_{1}\in\mathbb{R}$ satisfy $s_{0}<s<s_{1}$, and $\varphi\in\mathcal{M}$.

It follows from the definition of the space $\mathcal{Q}^{s-2,s/2-1;\varphi}$ that we need a relevant interpolation formula for this space in the case where $s\notin E$, with $E$ being  denoted by \eqref{setE}. Let $\{J_l:1\leq l\in\mathbb{Z}\}$ stand for the collection of all connected components of the set $(2,\infty)\setminus E$. Each  component $J_l$ is a certain finite subinterval of $(2,\infty)$.

\begin{lemma}\label{25lem7.4}
Let $1\leq l\in\mathbb{Z}$. Assume that real numbers $s_0,s,s_1\in J_{l}$ satisfy $s_0<s<s_1$ and that $\varphi\in\mathcal{M}$. Define an interpolation parameter $\psi$ by \eqref{8f16}.
Then
\begin{equation}\label{25f42}
\mathcal{Q}^{s-2,s/2-1;\varphi}=\,
\bigl[\mathcal{Q}^{s_0-2,s_0/2-1},
\mathcal{Q}^{s_1-2,s_1/2-1}\bigr]_{\psi}
\end{equation}
up to equivalence of norms.
\end{lemma}

\begin{proof}
We first prove an analog of formula \eqref{25f42} for the space $\mathcal{H}^{s-2,s/2-1;\varphi}$.
According to \cite[Theorem~2]{Los16JMathSci}, \cite[Lemma~2]{Los16JMathSci} (with $\Omega$ instead $\Pi$) and \cite[Theorems~1.5, 1.14(i), 3.2]{MikhailetsMurach14}, we have
\begin{align*}
\bigl[&\mathcal{H}^{s_0-2,s_0/2-1},
\mathcal{H}^{s_1-2,s_1/2-1}\bigr]_{\psi} \\
&=\bigl(\bigl[H^{s_0-2,s_0/2-1}(\Omega),H^{s_1-2,s_1/2-1}(\Omega)\bigr]_{\psi}\bigr)^{N}\\
&\qquad
\oplus\bigoplus_{j=1}^{N}
\bigl[H^{s_0-l_j-1/2,\,(s_0-l_j-1/2)/2}(S),
H^{s_1-l_j-1/2,\,(s_1-l_j-1/2)/2}(S)\bigr]_{\psi}\\
&\qquad
\oplus\bigl(\bigl[H^{s_0-1}(G),H^{s_1-1}(G)\bigr]_{\psi}\bigr)^{N}\\
&=\bigl(H^{s-2,s/2-1;\varphi}(\Omega)\bigr)^{N}
\oplus\bigoplus_{j=1}^{N}H^{s-l_j-1/2,(s-l_j-1/2)/2;\varphi}(S)
\oplus\bigl(H^{s-1;\varphi}(G)\bigr)^{N}\\
&=\mathcal{H}^{s-2,s/2-1;\varphi}.
\end{align*}
Thus,
\begin{equation}\label{25f45}
\bigl[\mathcal{H}^{s_0-2,s_0/2-1},
\mathcal{H}^{s_1-2,s_1/2-1}\bigr]_{\psi}=
\mathcal{H}^{s-2,s/2-1;\varphi}
\end{equation}
up to equivalence of norms.

Consider the set
\begin{equation}\label{set-kj}
\biggl\{k\in\mathbb{Z}:0\leq k<\frac{1}{2}\,\biggl(s-l_j-\frac{3}{2}\biggr)\biggr\}
\end{equation}
for each $j\in\{1,\dots,N\}$. This set relates to the compatibility conditions \eqref{25f8} and does not depend on $s$ whenever $s$ ranges over an arbitrary chosen set $J_l$. Let $q_{l,j}$ denote the number of all elements of
\eqref{set-kj}.

Letting $1\leq l\in\mathbb{Z}$, we define a linear mapping $P_l$ on
\begin{equation}\label{16f43-aa}
\bigcup_{\sigma\in J_{l}}\mathcal{H}^{\sigma-2,\sigma/2-1}
\end{equation}
such that the restriction of $P_l$ to the space
$\mathcal{H}^{\sigma-2,\sigma/2-1}$ is a projector of this space on its subspace $\mathcal{Q}^{\sigma-2,\sigma/2-1}$ for every
$\sigma\in J_{l}$. For any vector
$$
F:=\bigl(f_1,\dots,f_N,g_1,\dots,g_N,h_1,\dots,h_N\bigr)\in
\bigcup_{\sigma\in J_{l}}\mathcal{H}^{\sigma-2,\sigma/2-1},
$$
we put
\begin{equation}\label{16f43}
\left\{
  \begin{array}{ll}
    g^{*}_{j}:=g_j & \hbox{whenever}\; q_{l,j}=0,\\
    g^{*}_{j}:=g_j+T_{l,j}(w_{j,0},\dots,w_{j,q_{l,j}-1})
    & \hbox{whenever}\;q_{l,j}\geq1
  \end{array}
\right.
\end{equation}
for each $j\in\{1,\dots,N\}$. Here,
\begin{align*}
w_{j,0}&:=\mathcal{B}_{j,0}(v_{1,0},\dots,v_{N,0})\!\upharpoonright\!\Gamma-g_j\!\upharpoonright\!\Gamma,\\
&\dots\\
w_{j,q_{l,j}-1}&:=\mathcal{B}_{j,q_{l,j}-1}(v_{1,0},\dots,v_{N,0},\dots,v_{1,q_{l,j}-1},\dots,v_{N,q_{l,j}-1})\!\upharpoonright\!\Gamma\\
&\,-\partial_{t}^{q_{l,j}-1} g_j\!\upharpoonright\!\Gamma,
\end{align*}
with the functions $v_{1,0},\dots,v_{N,q_{l,j}-1}$ and the differential operators $\mathcal{B}_{j,0},\ldots,\mathcal{B}_{j,q_{l,j}-1}$ being defined
by \eqref{25f9} and \eqref{25f9B} respectively and with $T_{l,j}$ denoting the linear mapping $T$ from
\cite[Lemma~6.1]{LosMurach17OpenMath} in the $r=q_{l,j}$ case.

The linear mapping
\begin{align*}
P_l&:\,\bigl(f_1,\dots,f_N,g_1,\dots,g_N,h_1,\dots,h_N\bigr)\\
   &\mapsto\bigl(f_1,\dots,f_N,g^*_1,\dots,g^*_N,h_1,\dots,h_N\bigr)
\end{align*}
given on \eqref{16f43-aa} is required.
Indeed, its restriction to the space $\mathcal{H}^{\sigma-2,\sigma/2-1}$ is a bounded operator on this space for every $\sigma\in J_l$, which follows from
\eqref{25f9}--\eqref{25f69aa} and \cite[Lemma~6.1]{LosMurach17OpenMath}. We use the boundedness of the operators $T$ from
\cite[Lemma~6.1]{LosMurach17OpenMath} in the case where
$r=q_{l,j}$, $s=\sigma-l_j-1/2$, and $\varphi(\cdot)\equiv1$,
with the condition $s>2r-1$ being satisfied because
\begin{equation*}
q_{l,j}<\frac{1}{2}\biggl(\sigma-l_j-\frac{3}{2}\biggr)+1
\quad\mbox{whenever}\quad\sigma\in J_l.
\end{equation*}
Besides, it follows from the definitions of $P_l$ and the space $\mathcal{Q}^{\sigma-2,\sigma/2-1}$
that
$$P_lF=(f_1,\dots,f_N,g^*_1,\dots,g^*_N,h_1,\dots,h_N)\in\mathcal{Q}^{\sigma-2,\sigma/2-1}$$
for every
$F\in\mathcal{H}^{\sigma-2,\sigma/2-1}$.
Indeed,
\begin{align*}
\partial^{r}_t g_j^*\!\upharpoonright\!\Gamma=&
\partial^{r}_t g_j\!\upharpoonright\!\Gamma+
\partial^{r}_t T_{l,j}(w_{j,0},\dots,w_{j,q_{l,j}-1})\!\upharpoonright\!\Gamma\\
=& \partial^{r}_t g_j\!\upharpoonright\!\Gamma+w_{j,r}\\
=& \partial^{r}_t g_j\!\upharpoonright\!\Gamma+
\mathcal{B}_{j,r}(v_{1,0},\dots,v_{N,0},\dots,v_{1,r},\dots,v_{N,r})\!\upharpoonright\!\Gamma-
\partial^{r}_t g_j\!\upharpoonright\!\Gamma\\
=&\mathcal{B}_{j,r}(v_{1,0},\dots,v_{N,0},\dots,v_{1,r},\dots,v_{N,r})\!\upharpoonright\!\Gamma
\end{align*}
for all $j$ and $r$ indicated in \eqref{25f8}.
Thus, the vector $P_{l}F$ satisfies the compatibility conditions, i.e. $P_{l}F\in\mathcal{Q}^{\sigma-2,\sigma/2-1}$.
Moreover, $F\in\mathcal{Q}^{\sigma-2,\sigma/2-1}$ implies that
$P_lF=F$. Namely, if $F\in\mathcal{Q}^{\sigma-2,\sigma/2-1}$, then
\eqref{25f8} holds, which entails that all $w_{j,r}=0$, i.e.
$g_1^*=g_1$, ..., $g_N^*=g_N$.

According to \cite[Theorem~1.6]{MikhailetsMurach14}, the pair
$$
\bigl[\mathcal{Q}^{s_0-2,s_0/2-1},
\mathcal{Q}^{s_1-2,s_1/2-1}\bigr]
$$
is regular, and
\begin{equation}\label{25f44}
\begin{aligned}
&\bigl[\mathcal{Q}^{s_0-2,s_0/2-1},
\mathcal{Q}^{s_1-2,s_1/2-1}\bigr]_{\psi}\\
&=\bigl[\mathcal{H}^{s_0-2,s_0/2-1},
\mathcal{H}^{s_1-2,s_1/2-1}\bigr]_{\psi}\cap
\mathcal{Q}^{s_0-2,s_0/2-1}
\end{aligned}
\end{equation}
up to equivalence of norms. Formulas \eqref{25f44} and \eqref{25f45} give
\begin{align*}
&\bigl[\mathcal{Q}^{s_0-2,s_0/2-1},
\mathcal{Q}^{s_1-2,s_1/2-1}\bigr]_{\psi}\\
&=\mathcal{H}^{s-2,s/2-1;\varphi}\cap
\mathcal{Q}^{s_0-2,s_0/2-1}=
\mathcal{Q}^{s-2,s/2-1;\varphi}.
\end{align*}
The latter equality holds true because $s,s_0\in J_l$; i.e., all
elements of the subspace $\mathcal{Q}^{s-2,s/2-1;\varphi}$
satisfy the same compatibility conditions as elements of
$\mathcal{Q}^{s_0-2,s_0/2-1}$.

\end{proof}

\begin{proof}[The proof of Theorem~$\ref{25th4.1}$.]
Let $s>2$ and $\varphi\in\mathcal{M}$. We first consider the case where $s\notin E$. Then $s\in J_{l}$ for some integer $l\geq1$.
Choose numbers $s_0,s_1\in J_{l}$ such that $s_0<s<s_1$ and that  $s_j+1/2\notin\mathbb{Z}$ and $s_j/2+1/2\notin\mathbb{Z}$ whenever $j\in\{0,1\}$.
According to Zhitarashu and Eidelman  \cite[Theorem~5.7]{ZhitarashuEidelman98}, the mapping \eqref{25f4.1}
extends uniquely (by continuity) to an isomorphism
\begin{equation}\label{25f49}
\Lambda:\,\bigl(H^{s_j,s_j/2}(\Omega)\bigr)^N\leftrightarrow
\mathcal{Q}^{s_j-2,s_j/2-1}
\quad\mbox{for each}\quad j\in\{0,1\}.
\end{equation}

We define
the interpolation parameter $\psi$ by \eqref{8f16}. It follows from \cite[Theorem~1.5]{MikhailetsMurach14}, \cite[Lemma~2]{Los16JMathSci} (with $\Omega$ instead $\Pi$) and
Lemma~\ref{25lem7.4} that that the restriction of the operator \eqref{25f49} with $j=0$ to the space
\begin{equation*}
\bigl[\bigl(H^{s_0,s_{0}/2}(\Omega)\bigr)^N,
\bigl(H^{s_1,s_{1}/2}(\Omega)\bigr)^N\bigr]_{\psi}
=\bigl(H^{s,s/2;\varphi}(\Omega)\bigr)^N
\end{equation*}
sets an isomorphism
\begin{equation*}\label{25f50}
\Lambda:\,\bigl(H^{s,s/2;\varphi}(\Omega)\bigr)^N \leftrightarrow
\bigl[\mathcal{Q}^{s_0-2,s_0/2-1},
\mathcal{Q}^{s_1-2,s_1/2-1}\bigr]_{\psi}
=\mathcal{Q}^{{s-2,s/2-1};\varphi}.
\end{equation*}
Since the set $(C^{\infty}(\overline{\Omega}))^N$ is dense in $(H^{s,s/2;\varphi}(\Omega))^N$, the operator \eqref{25f50} is an extension by continuity of \eqref{25f4.1}.
Hence, Theorem \ref{25th4.1} is proved in the $s\notin E$ case.

Now we consider the case $s\in E$.
Choose $\varepsilon\in(0,1/2)$ arbitrarily.
Since $s\pm\varepsilon\notin E$ and $s-\varepsilon>2$, we have the isomorphisms
\begin{equation}\label{8f37}
\Lambda:\bigl(H^{s\pm\varepsilon,(s\pm\varepsilon)/2;\varphi}(\Omega)\bigr)^N \leftrightarrow
\mathcal{Q}^{s\pm\varepsilon-2,(s\pm\varepsilon)/2-1;\varphi}
\end{equation}
as has just been proved. They imply that the mapping \eqref{25f4.1} extends uniquely (by continuity) to an isomorphism
\begin{equation}\label{8f38}
\begin{aligned}
&\Lambda:
\bigl[\bigl(H^{s-\varepsilon,(s-\varepsilon)/2;\varphi}(\Omega)\bigr)^N,
\bigl(H^{s+\varepsilon,(s+\varepsilon)/2;\varphi}(\Omega)\bigr)^N\bigr]_{1/2}\\
&\leftrightarrow
\bigl[\mathcal{Q}^{s-\varepsilon-2,(s-\varepsilon)/2-1;\varphi},
\mathcal{Q}^{s+\varepsilon-2,(s+\varepsilon)/2-1;\varphi}\bigr]_{1/2}=
\mathcal{Q}^{s-2,s/2-1;\varphi}.
\end{aligned}
\end{equation}
The last equality is the definition of the space $\mathcal{Q}^{s-2m,(s-2m)/(2b);\varphi}$.
Owing to \cite[Formula (61)]{LosMurach17OpenMath}, we have
\begin{equation}\label{8f38A}
\bigl[\bigl(H^{s-\varepsilon,(s-\varepsilon)/2;\varphi}(\Omega)\bigr)^N,
\bigl(H^{s+\varepsilon,(s+\varepsilon)/2;\varphi}(\Omega)\bigr)^N\bigr]_{1/2}\\
=\bigl(H^{s,s/2;\varphi}(\Omega)\bigr)^N.
\end{equation}
It remains to apply \eqref{8f38A} to \eqref{8f38}.
\end{proof}

\begin{remark}\label{25rem8.1}
Let $s\in E$. The space $\mathcal{Q}^{s-2,s/2-1;\varphi}$ defined by \eqref{25f10} does not depend on $\varepsilon\in(0,1/2)$ up to equivalence of norms.
Indeed, according to Theorem \ref{25th4.1}, the isomorphism
\begin{equation*}
\Lambda:\bigl(H^{s,s/2;\varphi}(\Omega)\bigr)^N\leftrightarrow
\bigl[\mathcal{Q}^{s-\varepsilon-2,(s-\varepsilon)/2-1;\varphi},
\mathcal{Q}^{s+\varepsilon-2,(s+\varepsilon)/2-1;\varphi}\bigr]_{1/2}
\end{equation*}
holds true whenever $0<\varepsilon<1/2$. This directly implies the mentioned independence.
Besides, the space $\mathcal{Q}^{s-2,s/2-1;\varphi}$ is embedded continuously in $\mathcal{H}^{s-2,s/2-1;\varphi}$.
Indeed, choosing $\varepsilon\in(0,1/2)$, we get the continuous embeddings
\begin{equation}\label{25f69}
\mathcal{Q}^{s\mp\varepsilon-2,(s\mp\varepsilon)/2-1;\varphi}
\hookrightarrow
\mathcal{H}^{s\mp\varepsilon-2,(s\mp\varepsilon)/2-1;\varphi}
\end{equation}
in view of $s\mp\varepsilon\in(2,\infty)\setminus E$ and the definition of the left-hand space.
Owing to the interpolation formula \eqref{8f38A} and its analogs for the spaces on $S$ and $G$, we obtain
\begin{equation}\label{25f67}
\mathcal{H}^{s-2,s/2-1;\varphi}=
\bigl[\mathcal{H}^{s-\varepsilon-2,(s-\varepsilon)/2-1;\varphi},
\mathcal{H}^{s+\varepsilon-2,(s+\varepsilon)/2-1;\varphi}\bigr]_{1/2}
\end{equation}
(see also \cite[Lemma 6.4]{LosMikhailetsMurach21arXiv}).
It follows from \eqref{25f69} that the embedding operator acts continuously from \eqref{25f10} to \eqref{25f67}, as was stated.
\end{remark}


\begin{thebibliography}{99}

\bibitem{AgranovichVishik64}
M. S. Agranovich, M. I. Vishik, \emph{Elliptic problems with parameter and parabolic
problems of general form}, Russian Math. Surveys \textbf{19} (1964), no.~3, 53--157.

\bibitem{AnopKasirenko16MFAT}
A. V. Anop, T. M. Kasirenko, \emph{Elliptic boundary-value problems in H\"ormander spaces}, Methods Funct. Anal. Topology \textbf{22} (2016),
no.~4, 295--310.

\bibitem{AnopDenkMurach21CPAA}
A. Anop, R. Denk, A. Murach, \emph{Elliptic problems with rough boundary data in generalized Sobolev spaces}, Commun. Pure Appl. Anal. \textbf{20} (2021), no.~2, 697--735.

\bibitem{BinghamGoldieTeugels89}
N. H. Bingham, C. M. Goldie, J. L. Teugels, \emph{Regular Variation},
Encyclopedia Math. Appl., 27, Cambridge University Press, Cambridge, 1989.

\bibitem{DenkHieberPruess07}
Denk R., Hieber M., Pr\"uess J.  \emph{Optimal $L_p-L_q$-estimates for parabolic boundary value problems with inhomogeneous data.}
Mathematische Zeitschrift. 2007. Vol. 257. no.~1. P. 193--224.

\bibitem{DongKim15}
Dong H., Kim D. \emph{Elliptic and parabolic equations with measurable coefficients in weighted Sobolev spaces.}
Advances in Mathematics. 2015. Vol.~274. P. 681--735.

\bibitem{Eidelman94}
S. D. Eidel'man, \emph{Parabolic equations}, Encyclopaedia Math. Sci.,
(Partial differential equations, VI. Elliptic and parabolic operators),
Springer, Berlin, vol.~63, 1994, pp.~205--316.

\bibitem{ZhitarashuEidelman98}
S. D. Eidel'man, N. V. Zhitarashu,
\emph{Parabolic Boundary Value Problems},
Operator Theory: Advances and Applications, vol.~101, Birkh\"aser, Basel, 1998.

\bibitem{Hermander63}
L. H\"ormander, \emph{Linear Partial Differential Operators}, Grundlehren Math.
Wiss., Band~116, Springer, Berlin, 1963.

\bibitem{Hermander83}
L. H\"ormander, \emph{The Analysis of Linear Partial Differential Operators. Vol.II. Differential Operators with Constant Coefficients}, Springer-Verlag, Berlin, 2005.

\bibitem{Hummel21}
Hummel F. \emph{Boundary value problems of elliptic and parabolic type with boundary data of negative regularity.}
Journal of Evolution Equations (2021). https://doi.org/10.1007/s00028-020-00664-0

\bibitem{Jacob010205}
N. Jacob, \emph{Pseudodifferential Operators and Markov Processes}, in 3 volumes,
Imperial College Press, London, 2001, 2002, 2005.

\bibitem{Karamata30a}
J. Karamata, \emph{Sur certains "Tauberian theorems"\;de M.~M.~Hardy et Littlewood}, Mathematica (Cluj),
\textbf{3} (1930), 33--48.

\bibitem{LadyzhenskajaSolonnikovUraltzeva67}
O. A. Lady\v{z}enskaja, V. A. Solonnikov, N. N. Ural'tzeva, \emph{Linear and
Quasilinear Equations of Parabolic Type}, Transl. Math. Monogr., vol.~23, American
Mathematical Society, Providence, R.I., 1967.

\bibitem{Lindemulder20}
Lindemulder N. \emph{Maximal regularity with weights for parabolic problems with inhomogeneous boundary conditions.}
Journal of Evolution Equations. 2020. Vol.~20. no.~1. P. 59--108.

\bibitem{LionsMagenes72ii}
J.-L. Lions, E. Magenes, \emph{Non-Homogeneous Boundary-Value Problems and
Applications, vol.~II}, Grundlehren Math. Wiss., Band~182, Springer, Berlin, 1972.

\bibitem{Los16JMathSci}
V. M. Los, \emph{Anisotropic Hormander Spaces on the Lateral Surface of a Cylinder},
J. Math. Sci. \textbf{217} (2016), no.~4, 456 -- 467.

\bibitem{Los16UMJ6}
V. M. Los, \emph{Theorems on Isomorphisms for Some Parabolic Initial-Boundary-Value Problems
in Hormander Spaces: Limiting Case},
Ukrainian Math. J. \textbf{68} (2016), no.~6, 894--909.

\bibitem{Los16UMJ9}
V. M. Los, \emph{Classical solutions of the
parabolic initial-boundary value problems and  Hormander spaces},
Ukrainian Math. J. \textbf{68} (2016), no.~9, 1229-1239.

\bibitem{Los16UMJ11}
V. M. Los, \emph{Sufficient conditions for the solutions of general parabolic initial-boundary-value problems to be classical},
Ukrainian Math. J. \textbf{68} (2017), no.~11, 1756--1766.

\bibitem{Los17UMJ3}
V. M. Los, \emph{Systems parabolic in Petrovskii's sense in H\"ormander Spaces},
Ukrainian Math. J. \textbf{69} (2017), no.~3, 426--443.

\bibitem{Los20MFAT2}
V. Los, \emph{A condition for generalized solutions of a parabolic problem for a Petrovskii system to be classical},
Methods Funct. Anal. Topology, \textbf{26} (2020), no.~2, 111--118.

\bibitem{LosMikhailetsMurach17CPAA}
V. Los, V. A. Mikhailets, A. A. Murach, \emph{An isomorphism theorem for parabolic problems
in H\"ormander spaces and its applications},
Commun. Pur. Appl. Anal, \textbf{16} (2017), no.~1, 69--97.

\bibitem{LosMikhailetsMurach21arXiv}
V. Los, V. A. Mikhailets, A. A. Murach, \emph{Parabolic problems in generalized Sobolev spaces}, (2019), arXiv:1907.04283

\bibitem{LosMurach17OpenMath}
V. Los, A. A. Murach, \emph{Isomorphism theorems for some parabolic initial-boundary
value problems in H\"ormander spaces},
Open Mathematics \textbf{15} (2017), 57--76.

\bibitem{MikhailetsMurach12BJMA2}
V. A. Mikhailets, A. A. Murach, \emph{The refined Sobolev scale, interpolation, and
elliptic problems}, Banach J. Math. Anal. \textbf{6} (2012), no.~2, 211--281.

\bibitem{MikhailetsMurach14}
V. A. Mikhailets, A. A. Murach, \emph{Ho\"rmander spaces, interpolation, and
elliptic problems}, De Gruyter, Berlin, 2014.

\bibitem{NicolaRodino10}
F. Nicola, L. Rodino, \emph{Global Pseudodifferential Calculas on Eucli\-dean
Spaces}, Birkh\"aser, Basel, 2010.

\bibitem{Paneah00}
B. Paneah, \emph{The Oblique Derivative Problem. The Poincar\'e Problem},
Wiley--VCH, Berlin, 2000.

\bibitem{Slobodetskii58}
L. N. Slobodeckii, \emph{
Generalized Sobolev spaces and their application to boundary problems for partial differential equations},
Leningrad. Gos. Ped. Inst. Uchen. Zap., \textbf{197} (1958), 54--112 (Russian).
[English translation in Amer. Math. Soc. Transl. \textbf{57} (1966), no.~2, 207--275.]

\bibitem{Solonnikov64}
V. A. Solonnikov, \emph{
Apriori estimates for solutions of second-order equations of parabolic type},
Trudy Mat. Inst. Steklov, \textbf{70} (1964), 133--212  (Russian).

\bibitem{Solonnikov65}
V. A. Solonnikov, \emph{On boundary value problems for linear parabolic systems of differential
equations of a general form}, Proc. Steklov Inst. Math., \textbf{83} (1965), 1--184.

\bibitem{Weidemaier05}
Weidemaier P.  \emph{Lizorkin-Triebel spaces of vector-valued functions and sharp trace theory for functions in Sobolev spaces with a mixed $L_p$-norm in parabolic problems.}
Sbornik: Mathematics. 2005. Vol. 196. no.~6. P. 3--16.

\end{thebibliography}
\end{document}